\documentclass{amsart}

\usepackage{graphicx,amssymb}
\usepackage[all]{xy}

\newtheorem{theorem}{Theorem}[section]
\newtheorem*{thmAAB}{Theorem (Abramovich-Aliprantis-Burkinshaw, 1993)}
\newtheorem*{thmG}{Theorem (Grivaux, 2002)}

\newtheorem*{remark*}{Remark}
\newtheorem{proposition}[theorem]{Proposition}
\newtheorem*{proposition*}{Proposition}
\newtheorem*{proposition-RT}{Proposition (Radjavi-Troitsky, 2008)}
\newtheorem{corollary}[theorem]{Corollary}

\theoremstyle{definition}
\newtheorem{definition}[theorem]{Definition}
\newtheorem{example}[theorem]{Example}

\theoremstyle{remark}

\numberwithin{equation}{section}

\allowdisplaybreaks

\theoremstyle{definition}

\newcommand{\N}{\mathbb{N}}

\newcommand{\al}{\alpha}

\newcommand{\sumn}{\sum_{n=1}^{\infty}}

\newcommand{\summ}{\sum_{m=1}^{\infty}}

\newcommand{\Fi}{\varphi}

\begin{document}

\title[Invariant subspaces for positive operators]{Invariant subspaces for positive operators on\\ Banach spaces with unconditional basis}

\author{Eva A. Gallardo-Guti\'{e}rrez}
\address{Eva A. Gallardo-Guti\'errez \newline Departamento de An\'alisis Matem\'atico y Matem\'atica Aplicada,\newline
Facultad de Matem\'aticas,
\newline Universidad Complutense de
Madrid, \newline
 Plaza de Ciencias N$^{\underbar{\Tiny o}}$ 3, 28040 Madrid,  Spain
 \newline
and Instituto de Ciencias Matem\'aticas (CSIC-UAM-UC3M-UCM),
\newline Madrid,  Spain } \email{eva.gallardo@mat.ucm.es}

\author{Javier Gonz\'alez-Do\~na}
\address{Javier Gonz\'alez-Do\~na \newline Departamento de An\'alisis Matem\'atico y Matem\'atica Aplicada,\newline
Facultad de Matem\'aticas,
\newline Universidad Complutense de
Madrid, \newline
 Plaza de Ciencias N$^{\underbar{\Tiny o}}$ 3, 28040 Madrid,  Spain
 \newline
and Instituto de Ciencias Matem\'aticas (CSIC-UAM-UC3M-UCM),
\newline Madrid,  Spain }
\email{javier.gonzalez@icmat.es}

\author{Pedro Tradacete}
\address{Pedro Tradacete \newline  Instituto de Ciencias Matem\'aticas (CSIC-UAM-UC3M-UCM),
\newline C/ Nicol\'as Cabrera, 13–15, \newline
Campus de Cantoblanco UAM,
\newline Madrid 28049,  Spain }
\email{pedro.tradacete@icmat.es}

\thanks{First two authors are partially supported by Plan Nacional  I+D grant no. MTM2016-77710-P, Spain. Second author also acknowledges support of the FPI Grant  PRE 2018-083669. Third author gratefully acknowledges support of Agencia Estatal de Investigaci\'on (AEI) and Fondo Europeo de Desarrollo Regional (FEDER) through grants MTM2016-76808-P (AEI/FEDER, UE) and MTM2016-75196-P (AEI/FEDER, UE).\newline
This work has been partially supported by Grupo UCM 910346 and the ICMAT Severo Ochoa project SEV-2015-0554 of the Ministry of Economy and
Competitiveness of Spain and by the European Regional Development
Fund.}


\date{May 2020}


\begin{abstract}
We prove that every lattice homomorphism acting on a Banach space $\mathcal{X}$ with the lattice structure given by an unconditional basis has a non-trivial closed invariant subspace. In fact, it has a non-trivial closed invariant ideal, which is no longer true for every positive operator on such a space. Motivated by these later examples, we characterize tridiagonal positive operators without non-trivial closed invariant ideals on $\mathcal{X}$ extending to this context a result of Grivaux on the existence of non-trivial closed invariant subspaces for tridiagonal operators. 
\end{abstract}

\maketitle

\section{Introduction}

Let $\mathcal{X}$ be an infinite dimensional separable (real or complex) Banach space and $\mathcal{E}=\{e_n\}_{n\geq 1}$ a Schauder basis (or  simply a  basis) of $\mathcal{X}$, that is, for  every $x\in \mathcal{X}$, there exists a  unique sequence of  scalars $\{\alpha_n\}$ such that
$$x=\sum_{n=1}^{\infty} \alpha_n e_n.$$
Clearly the basis $\{e_n\}_{n\geq 1}$ gives rise to  a natural closed cone $\mathcal{C}_{\mathcal{E}}$  defined by
$$
\mathcal{C}_{\mathcal{E}}= \left \{x=\sum_{n=1}^{\infty} \alpha_n e_n\, :\;  \alpha_n \geq 0 \mbox{ for each } n = 1, 2,\dots \right \}.
$$
Observe that $\mathcal{C}_{\mathcal{E}}$ satisfies trivially the properties of a cone, namely, $\mathcal{C}_{\mathcal{E}} + \mathcal{C}_{\mathcal{E}} \subseteq  \mathcal{C}_{\mathcal{E}}$, $a\, \mathcal{C}_{\mathcal{E}} \subseteq  \mathcal{C}_{\mathcal{E}}$ for each real $a>0$, and
$\mathcal{C}_{\mathcal{E}}\cap \left(-\mathcal{C}_{\mathcal{E}}\right )=\{0\}$.

In addition, it is well known that every cone $C$ in a Banach space $\mathcal{X}$ determines a partial order $\leq $ by letting $y \leq x$  whenever
$x-y\in C$.  The elements of $C$ are known as positive vectors and the pair $(\mathcal{X}, C)$ is an ordered Banach space.

The order structure plays an important role when linear operators acts on $\mathcal{X}$. At this regards, recall that a linear operator $T$ on $\mathcal{X}$  is said to be \emph{positive} respect to the basis $\mathcal{E}=\{e_n\}_{n\geq 1}$  if $T(\mathcal{C}_{\mathcal{E}})\subset \mathcal{C}_{\mathcal{E}}$, or  equivalently, $Tx\geq 0$  for each $x\geq 0$. By identifying $T$  with the infinite matrix $(a_{ij})_{\{i, j\geq 1\}}$ respect to the basis $\mathcal{E}$, $T$ is a positive operator if and only if $a_{ij}\geq 0$ for every pair $(i,\, j)$. If the
basis $\mathcal{E}=\{e_n\}_{n\geq 1}$ is also unconditional, then every positive operator is automatically
continuous; see \cite[Theorem 12.3]{AB-book}.

Nevertheless, despite of the order structure inherited in the Banach space with an unconditional basis, it is still unknown whether every positive operator (and therefore, a positive continuous operator) has a non-trivial closed invariant subspace. Indeed, motivated by the Invariant Subspace Problem, Abramovich, Aliprantis and Burkinshaw proved in the seminal work  \cite{AAB93} the following result for the classical Banach spaces of sequences $\ell^p$ ($1\leq p<\infty$).

\begin{thmAAB} Let $T$ be a positive operator acting on $\ell^p$, $(1\leq p<\infty)$. Assume there exists a positive operator $S$ in the commutant of $T$ which is locally quasinilpotent at a non-zero positive vector, that is, there exists a positive vector $x_0$ such that
$$
\lim_n \|S^n x_0\|^{1/n}=0.
$$
Then $T$ has a non-trivial complemented closed invariant subspace. Moreover, $T$ has a  non-trivial closed invariant ideal.
\end{thmAAB}

Later on, in \cite{AAB94}, the authors extended previous results on the existence of closed invariant subspaces to positive operators on Banach lattices or operators close to them. Shortly after, in \cite{AAB95}, they extended these results to operators acting on any Banach space $\mathcal{X}$ with a (not
necessarily unconditional) Schauder basis. In particular, they proved that if a
continuous quasinilpotent operator $T$ on $\mathcal{X}$ is positive with respect
to the closed cone generated by a basis, then the operator has a nontrivial closed
invariant subspace. Moreover, if $T$ commutes with a nonzero positive operator
that is quasinilpotent at a nonzero positive vector, then T has a nontrivial closed
invariant subspace.
\medskip

The role played by the (local) quasinilpotent behavior in the aforementioned results seems to be quite relevant, and there have been attempts to extend those results for positive operators (not necessarily (locally) quasinilpotent) by means of considering Lomonosov's theorem \cite{Lomonosov} (we refer to \cite[Chapter 10]{AA-book} for more on the subject and to \cite{Troitsky} for recent results at this regards). Indeed, if an ideal irreducible positive operator $T$, that is, a positive operator $T$ lacking non-trivial invariant ideals, commutes with a compact positive operator $S$, then neither $T$ nor $S$ are locally quasinilpotent at any non-zero vector (see
\cite[Problem 9.2.9]{AA-problembook}, for instance). Nevertheless, it is possible to exhibit a positive operator $T$ in a Banach space with an unconditional basis such that the only compact operator commuting with every non-scalar operator in the commutant of $T$ is the zero operator but still $T$ has non-trivial closed invariant ideals and fails to be quasinilpotent at any non-zero vector $x$ (see Section \ref{sec2}).
\medskip

A closer look at the example in Section \ref{sec2} shows that it is a lattice homomorphism. At this regards, in Section \ref{sec3}, we will show that every lattice homomorphism in any Banach space $\mathcal{X}$ with an unconditional Schauder basis has non-trivial closed invariant subspaces; indeed  non-trivial closed invariant ideals (see Theorem \ref{lattice invariant}). Lattice homomorphisms are positive operators; and clearly they constitute a first class of operators to understand if every positive operator (and therefore, a positive continuous operator) has a non-trivial closed invariant subspace. It is worthy to point out that not every positive operator in a Banach lattice has a non-trivial closed invariant ideal (see Example \ref{Example1}) which sharpens Theorem \ref{lattice invariant} at that regards.
\medskip

Indeed, as we will show, Example \ref{Example1} motivates Section \ref{sec4} in order to study positive operators on $\mathcal{X}$ whose matrix representation respect to $\mathcal{E}$ is tridiagonal. While such operators have non-trivial closed invariant subspaces as Sophie Grivaux proved in \cite{G02}, we will characterize those lacking non-trivial closed invariant ideals.

\section{An example}\label{sec2}

In this section, we exhibit an example of an operator $T$  with non-trivial closed invariant ideals in a Banach space with an unconditional basis such that every non-scalar operator in the commutant of $T$ does not commute with any non-zero compact operator and $T$ is also not locally quasinilpotent at any non-zero vector. This result, based on the ideas in \cite{HNRR}, establishes that either the local quasinilpotent behavior of $T$ or the fact of having a non-zero compact operator in the commutant of $T$ are just sufficient conditions in order to draw the existence of non-trivial closed invariant ideals for positive operators.
\medskip

Let $\beta=\{\beta_n\}_{n\geq 0}$ be a sequence of positive
numbers with $\beta_0=1$ and consider $H^2(\beta)$  the set of all
formal power series $f(z)=\sum_{n=0}^{\infty} a_n z^n$ such that the series
$$
\sum_{n=0}^{\infty} |a_n|^2 \beta_n^2
$$
converges. The \emph{shift operator} $M_z$ defined by
$$\sum_{n=0}^{\infty} a_n z^n \to \sum_{n=0}^{\infty} a_n z^{n+1}$$
is a bounded linear transformation mapping $H^2(\beta)$ into itself under appropriate conditions
on the sequence $\beta$. Indeed, it is well-known that every weighted shift acting on $\ell^2$ is unitarily equivalent to
$M_z$ acting on a suitable $H^2(\beta)$ (for more on the subject, we refer to the classical survey by Shields \cite{Shields}).
\medskip

In \cite{HNRR}, Hadwin, Nordgren, Radjavi and Rosenthal proved that quasi-analytic shifts  (see \cite[p. 103]{Shields} for definition) do not satisfy Lomonosov's hypothesis. In particular, by considering the sequence $\beta_n=\exp\left(n^{1/2}\right)$ the operator $M_z$ acting on $H^2(\beta)$ satisfies that the only compact operator commuting with every non-scalar operator in the commutant $\{M_z\}^\prime$  is the zero operator. We observe that $M_z$ is unitarily equivalent to the weighted shift $W$ on $\ell^2$ associated to the weight sequence $w_n= \exp\left((n+1)^{1/2}- n^{1/2}\right)$, namely
\begin{eqnarray} \label{def W}
W:\ell^2 &\rightarrow &\ell^2 \nonumber\\
 e_n & \rightarrow & w_n e_{n+1},
\end{eqnarray}
where $\{e_n\}_{n\geq 0}$ denotes the canonical basis in $\ell^2$. So, the only compact operator commuting with every non-scalar operator in $\{W\}^\prime$  is the zero operator. Clearly, $W$ has infinitely many invariant closed ideals, and moreover, it satisfies the following

\begin{proposition}\label{prop1}
The weighted shift $W$ acting on $\ell^2$ given by (\ref{def W}) is not quasinilpotent at any non-zero vector $x\in \ell^2$. Moreover, any operator commuting with $W$ is not quasinilpotent at any non-zero vector $x\in \ell^2$.
\end{proposition}

\begin{proof}
Since $W$ is unitarily equivalent to $M_z$ acting on $H^2(\beta)$, we will show that $M_z$ is not locally quasinilpotent in  $H^2(\beta)$ where $\beta_n=\exp\left(n^{1/2}\right)$. To that end,  observe that $H^2(\beta)$ consists of analytic functions in the open unit disc $\mathbb{D}$ of the complex plane which, in particular, are continuous on the closure $\overline{\mathbb{D}}$.

First, observe that for every non-zero $f\in H^2(\beta)$, the limit
$$
r_{M_z}(f):=\limsup_n  \left \|M_z^n f\right \|_{H^2(\beta)}^{1/n}
$$
is the local spectral radius of $M_z$ at $f$ in $H^2(\beta)$. Now, recall that the local spectrum  of $M_z$ at $f$ in $H^2(\beta)$, denoted by $\sigma_{M_z}(f)$, is the complement of the
set of complex numbers $\lambda$  such that there exists an open neighborhood $U_\lambda \ni \lambda$ and an analytic function $h:U_\lambda\rightarrow  H^2(\beta)$ for which
\begin{equation*}
(M_z-wI) h(w) = f, \quad \text{for every } w \in U_\lambda.
\end{equation*}
We refer to the monograph by Laursen and Neumann \cite{LAURSEN-NEUMANN_book} for more on local spectrum of operators. In particular, it holds that
$$
\mathbb{D}\subseteq \sigma_{M_z}(f)
$$
for all non-zero $f\in H^2(\beta)$ (see \cite[Proposition 1.6.9]{LAURSEN-NEUMANN_book}, for instance). Since for every linear bounded operator $T$ on a Banach space $X$ the local spectral radius at $x$ satisfies
$$
r_T(x)\geq \{ |\lambda|:\; \lambda \in \sigma_T(x) \},
$$
one deduces that $\limsup_n  \left \|M_z^n f\right \|_{H^2(\beta)}^{1/n}\geq 1$ for every non-zero $f\in H^2(\beta)$. Hence, $M_z$ is not quasinilpotent at any non-zero vector $f\in H^2(\beta)$.

\medskip

Note that the last statement of Proposition \ref{prop1} follows since every operator in the commutant of $M_z$ in $H^2(\beta)$ is
a multiplication operator $M_{\phi}$ where $\phi \in H^2(\beta)$ (see \cite{HNRR}). The proof is then completed by arguing, in a similar way, since the local spectral radius at any vector satisfies $\phi(\mathbb{D})\subseteq \sigma_{M_\phi}(f)$.
\end{proof}

\section{Invariant subspaces for lattice homomorphisms}\label{sec3}

Throughout this section, $\mathcal{X}$ will denote an infinite dimensional separable (real or complex) Banach space and $\mathcal{E}=\{e_n\}_{n\geq 1}$ an unconditional basis in $\mathcal{X}$. We consider the order induced by the basis $\mathcal{E}$ as mentioned in the Introduction.

Following the standard lattice notation, the supremum (least upper bound) and the infimum (greatest lower bound) of a pair of vectors  $x,\, y \in \mathcal{X}$ will be denoted by $x\wedge y$ and $x\vee y$ respectively, namely
$$x\vee y=\sup \{x, y\} \mbox{ and } x\wedge y=\inf \{x, y\}.$$

Recall that for $x$ in a vector lattice, its \emph{positive part}, its \emph{negative part} and its \emph{absolute value} are defined by:
$$
x^+=x\vee 0, \; \;  x^-=(-x)\vee 0, \; \; \mbox{ and } |x|=x\vee (-x),
$$
respectively. A sublattice of a vector lattice is a subspace which is also closed under the lattice operations (that is, for each $x,y$ in a sublattice, $x\vee y$ and $x\wedge y$ also belong to the sublattice). An ideal $M$ in a vector lattice is a subspace such that for every $x,\, y$ such that if $|x|\leq |y|$ and $y\in M$ then $x\in M$. Recall that a norm in a vector lattice is a lattice norm if it satisfies $\|x\|\leq \|y\|$ whenever $|x|\leq |y|$. Thus, a Banach lattice is a complete vector lattice equipped with a lattice norm.

In addition, recall that a linear bounded operator $T$ on a vector lattice $E$ is a \emph{lattice homomorphism} if $T(x\vee y)=Tx\vee Ty$ for all $x,\, y \in E$ (see also  \cite[Theorem 1.34]{AA-book}). Every lattice homomorphism is clearly a positive operator. Finally, a one-to-one lattice homomorphism which is onto is called a \emph{lattice isomorphism}.

A first result identifies lattice isomorphisms acting on Banach spaces $\mathcal{X}$ (real or complex) with an unconditional basis.

\begin{proposition}\label{representacion}
Let $\mathcal{X}$ be a Banach space with an unconditional basis $\mathcal{E}=\{e_n\}_{n\geq 1}$ and $T$ a lattice isomorphism. Then, there exists a permutation $\xi : \N \rightarrow \N$ and a sequence of positive numbers $\{w_n\}_{n\geq 1}$ such that
\begin{equation}\label{permutacion con pesos}
T e_n = w_n  e_{\xi(n)}, \qquad (n\in \N).
\end{equation}
\end{proposition}

\begin{proof}
Let $n, m \in \N$ such that $n\neq m$. First, observe that $e_n \wedge e_m = 0$. Since $T$ is a lattice homomorphism, it follows that $Te_n \wedge Te_m = T(e_n\wedge e_m) = 0.$ Denote $$ Te_n = \sum_{k \in N} \al_{k} e_k, \qquad Te_m = \sum_{k \in M} \beta_k e_k,$$ where $N,M \subset \N$ and $\al_k > 0$ and $\beta_k > 0$.
Since $Te_n$ and $Te_m$ are disjoint elements in $\mathcal{X}$, it follows that $N \cap M = \emptyset.$
Accordingly, every lattice homomorphism in $\mathcal{X}$ maps different basis elements into expansions with disjoint basis elements.

Now, note that $T^{-1}$ is also a lattice homomorphism since for every $x, y \in \mathcal{X}$
$$ x\vee y = T(T^{-1}x) \vee T(T^{-1}y) = T( (T^{-1}x)\vee T^{-1}y),$$ from where it easily yields the claim on $T^{-1}$. Hence,
one deduces
$$e_n = T^{-1}Te_n = \sum_{k \in N} \al_k T^{-1}e_k = \sum_{k \in N} \al_k \left( \sum_{l \in N_k} \gamma_l e_l \right),$$ where $\gamma_l > 0$ and $N_k \subset \N$.

Now, having in mind that the basis is unconditional, one deduces from the series expansion for $e_n$ that there exists $\eta \in \N$ such that $\al_\eta > 0$ and $\al_k = 0$ for every $k \in N,$ $k \neq \eta.$ Denote then by $\xi(n)$ such an  $\eta$ and $w_n = \al_{\xi(n)}.$ Hence
$$Te_n = w_n e_{\xi(n)}.$$
Upon applying such argument to every basis vector, it follows that $\xi : \N \rightarrow \N$ is a permutation since $T$ is bijective. From here the desired statement follows.
\end{proof}

\medskip

With Proposition \ref{representacion} at hand, we may introduce the following class of operators:

\begin{definition}
A linear bounded operator in a Banach space $\mathcal{X}$ with an unconditional basis $\mathcal{E}=\{e_n\}_{n\geq 1}$ will be called a  \textbf{weighted permutation operator} if there exists a permutation $\xi : \N \rightarrow \N$ and a sequence of positive numbers $\{w_n\}_{n\geq 1}$ such that $T e_n = w_n e_{\xi(n)}$ for every $n\in \N$.
\end{definition}

\medskip
Indeed, the previous ideas can be pushed a bit further and prove that injective, dense range lattice homomorphisms are also weighted permutation operators.

\begin{proposition}\label{representacion2}
Let $\mathcal{X}$ be a Banach space with an unconditional basis $\mathcal{E}=\{e_n\}_{n\geq 1}$ and $T$ an injective, dense range lattice homomorphism. Then $T$ is a weighted permutation operator.
\end{proposition}

\begin{proof}
Since $T$ is positive and injective,  for every $n \in \N$ there exist positive real numbers $\al_{n,k}$ and $N_n \subset \N$ such that
 $$ Te_n = \sum_{k \in N_n} \al_{n,k}\, e_k.$$
Assume, by contradiction, that there exists $\ell \in \N$ such that the cardinal of the set $N_{\ell}$, denoted by $|N_{\ell}|$, is strictly bigger than 1.
Let $n_0 \in N_{\ell}$ and write
$$Te_{\ell} =\alpha_{\ell, n_0}\, e_{n_0} + \sum_{k \in N_{\ell}\smallsetminus \{n_0\}} \alpha_{\ell, k}\, e_k.$$

Since the range of $T$ is dense, there exists a sequence $\{x_n\}_{n\geq 1}$ in $\mathcal{X}$ such that $Tx_n \rightarrow e_{n_0}$ in $\mathcal{X}$.
Write $x_n = \summ \beta_{n,m} e_m,$ and denote by $e^*_{n_0}$ the coordinate functional for $n_0 \in \N$. Having into account that $Te_n \wedge Te_m = 0$ for every $n\neq m \in \N$ (as in the proof of Proposition \ref{representacion}), it follows that
\begin{equation*}
e_{n_0}^*(Tx_n) = \beta_{n,\ell}\,  e^*_{n_0}\, (Te_{\ell}) = \beta_{n,\ell}\; \alpha_{\ell, n_0} \rightarrow 1, \qquad \mbox{ as  } n\to \infty.
\end{equation*}

Let $n_1 \in N_{\ell}$, with $n_1\neq n_0$. In a similar way, we have
\begin{equation*}
e_{n_1}^*(Tx_n) = \beta_{n,\ell}\,  e^*_{n_1}\, (Te_{\ell})= \beta_{n,\ell}\; \alpha_{\ell, n_1} \rightarrow 0, \qquad \mbox{ as  } n\to \infty,
\end{equation*}
which yields a contradiction. Hence, $|N_n| = 1$ for every $n \in \N$ and therefore $T$ is a weighted permutation operator.
\end{proof}

Next result will provide a matrix expression for lattice homomorphisms in Banach spaces with unconditional basis.

\begin{proposition}\label{matrix representation}
Let $\mathcal{X}$ be a Banach space with an unconditional basis $\mathcal{E}=\{e_n\}_{n\geq 1}$. Let $T$ be a positive operator on $\mathcal{X}$ and let $A = (a_{n,m})_{n,m \in \N}$ be the infinite positive matrix induced by $T$; namely $$Te_m = \sumn a_{n,m}e_n,$$ where $a_{n,m} \geq 0$ for every $n,m \in \N.$ Then, $T$ is a lattice homomorphism if and only if each row of $A$ has at most one non-zero entry. In addition, $T$ is an injective lattice homomorphism if and only if $A$ has no null columns.
\end{proposition}

\begin{proof}
First, let us assume that $T$ is a lattice homomorphism and suppose, by contradiction, that there exists $N \in \N$ and $i \neq j \in \N$ such that $a_{N,i}\cdot a_{N,j} > 0.$ Then, observe that the $N$-th coordinate of the vectors $Te_i$ and $Te_j$ are given by the coefficients $a_{N,i}$ and $a_{N,j}$ respectively. Since $T$ is a lattice homomorphism $Te_i \wedge Te_j=0$, but the $N$-th coordinate of such a vector  is given by $a_{N,i} \wedge a_{N_j} > 0$, which yields a contradiction.

\medskip

Now, assume that each row of $A$ has at most one non-zero entry. Then, it is obvious that $Te_n \wedge Te_m = 0$ for every $n\neq m$. We claim that this property for $T$ is equivalent to be a lattice homomorphism. Having in mind that $T$ is a positive operator, it is enough to show that $x\wedge y = 0$ implies $Tx \wedge Ty = 0$ for every $x, y \in \mathcal{X}$ (see \cite[Theorem 1.34]{AA-book}).

Let us suppose that $x\wedge y = 0$ for $x, y \in \mathcal{X}$. Hence, we may write
$$x = \sum_{n \in N} x_n e_n \;  \mbox{ and  }  y = \sum_{n \in M} y_n e_n,$$ where $N, M \subset \N$ satisfy $N \cap M = 0$ and $x_n,\, y_n \neq 0$. Equivalently, the supports of $x$ and $y$ are disjoint. Since $T$ maps disjoint basis vectors to elements with disjoint supports, $Tx$ and $Ty$ have also disjoint supports. Accordingly, $Tx \wedge Ty = 0$ and therefore, $T$ is a lattice homomorphism.

Last statement of the theorem follows easily from Proposition \ref{representacion2}.
\end{proof}

\begin{remark*}
Observe that as a straightforward application of Theorem \ref{matrix representation} the example of Section \ref{sec2} is a lattice homomorphism. What it is not clear from the matrix representation is the fact that such an operator is not quasinilpotent at any non-zero vector $x\in \ell^2$.
\end{remark*}

The main result of this Section is the following theorem.

\begin{theorem}\label{lattice invariant}
Let $\mathcal{X}$ be a Banach space with an unconditional basis $\mathcal{E}=\{e_n\}_{n\geq 1}$ and $T$ a lattice homomorphism on $\mathcal{X}$. Then $T$ has a non-trivial closed invariant subspace. Moreover, $T$ has a non-trivial closed invariant ideal.
\end{theorem}

\begin{proof}
First, observe that if $T$ is not injective, then the kernel of $T$ is a non-trivial closed invariant subspace; in fact, it is a closed invariant ideal since $T$ is a lattice homomorphism.

\smallskip

If $T$ is injective and has dense range then, by means of Proposition \ref{representacion2}, $T$ is a weighted permutation operator.
Let us denote by  $\xi : \N \rightarrow \N$ the permutation associated to $T$.

\smallskip

For $k, n \in \N$, let us consider the subset of $\N$
\begin{equation*}
A_k(n)= \{ \xi^{-(k+m)}(n) \in \N : m \in \N \}=\{\xi^{-(k+1)}(n), \, \xi^{-(k+2)}(n),\, \xi^{-(k+3)}(n),\, \dots \}.
\end{equation*}
That is, $A_k(n)$ is the orbit of $n$ under the sequence of maps $\{\xi^{-(k+m)}\}_{m\in \N}$.

\smallskip

First, let us show that $A_k(n) \subsetneq \N$ arguing by contradiction. If $A_k(n) = \N$,  then there exists $m \in \N$ such that $\xi^{-(k+m)}(n) = n$. In particular, in such a case, the permutation $\xi^{-1} : \N \rightarrow \N$ would be cyclic and $A_k(n)$ finite which contradicts the assumption.

Let us consider the ideal
\begin{equation} \label{idealdefinicion0}
I_{A_k(n)} = \left \{ x \in \mathcal{X}:\, x=\sumn x_n e_n \mbox{ with } x_m = 0 \ \text{for every} \ m \in A_k(n) \right \}.
\end{equation}
Since $A_k(n) \subsetneq \N$ we deduce that $I_{A_k(n)}$ is a non-trivial ideal. Moreover, $I_{A_k(n)}$ is invariant under $T$: note that if $x \in I_{A_k(n)}$ the coordinates of $Tx$ are given by
$$(Tx)_{\xi^{-(k+m)+1}(n)}$$
which are null for every $m \in \N$. Accordingly, $Tx \in I_{A_k(n)}$ and therefore, every injective and dense range $T$ homomorphism has a non-trivial closed invariant ideal and hence, a non-trivial invariant-subspace.

\medskip

In order to finish the proof, we are reduced to prove the result for injective lattice homomorphisms $T$ with no dense range. The aim is to show the existence of non-trivial closed invariant ideals, since clearly the closure of the range of $T$ is a non-trivial closed invariant subspace.

\smallskip

Denote by $A = (a_{n,m})_{n,m \in \N}$ the infinite positive matrix induced by $T$.  Since $T$ is injective, every column of $A$ has a strictly positive element. Assume there exists $n_0$ such that the $n_0$-th row of $A$ is null. Then, the non-trivial ideal
$$I_{n_0} = \{ x \in \mathcal{X}:\, x=\sumn x_n e_n \mbox{ with }  x_{n_0}=  0 \}$$
is clearly closed and invariant under $T$.  Hence, we may assume that every row and column of $A$ is not null.

\smallskip

Now, note that for each $n \in \N$ there exists $N_n \subset \N$ such that
$$Te_n = \sum_{k \in N_n} a_{k,n}e_k,$$
where $a_{k,n} > 0$ for every $k \in N_n$. Indeed, $N_n \neq \emptyset$ for every $n \in \N$ and $N_n\cap N_m = \emptyset $ if $n\neq m$. Moreover,  $$\N = \bigcup_{n=1}^\infty N_n.$$
These properties allow us to define the following map $\Fi : \N \rightarrow \N$. For each $k \in \N$ there exists a unique $n \in \N$ such that $k  \in N_n$. We define $\Fi(k) = n.$ Observe that $\Fi$ is a well-defined, surjective map.

\smallskip

Now, for $k, n \in \N$ we define the subset of $\N$
\begin{equation*}
\tilde{A}_k(n) = \{ \Fi^{k+m}(n) : m \in \N\}=\{\Fi^{k+1}(n), \, \Fi^{k+2}(n),\, \Fi^{k+2}(n),\, \dots \}.
\end{equation*}
In this case, $\tilde{A}_k(n)$ is the orbit of $n$ under the sequence of maps $\{\Fi^{k+m}\}_{m\in \N}$.

\medskip

In a similar way as before, let us show that $\tilde{A}_k(n) \subsetneq \N$ arguing by contradiction. If $\tilde{A}_k(n) = \N$, then there would exist $m_0 \in \N$ such that $\Fi^{k+m_0}(n) = n$. Thus, $\Fi$ would be cyclic and $\tilde{A}_k(n)$ finite, which is again a contradiction.

\smallskip

Let us consider the ideal $I_{\tilde{A}_k(n)}$ associated to $\tilde{A}_k(n)$ as in \eqref{idealdefinicion0}, namely,
\begin{equation*} 
I_{\tilde{A}_k(n)} = \left \{ x \in \mathcal{X}:\, x=\sumn x_n e_n \mbox{ with } x_m = 0 \ \text{for every} \ m \in \tilde{A}_k(n) \right \}.
\end{equation*}

Once again $I_{\tilde{A}_k(n)}$ is a non-trivial ideal of $\mathcal{X}$  since $\tilde{A}_k(n) \subsetneq \N$. The proof of the Theorem will follow if we show that $I_{\tilde{A}_k(n)}$ is invariant under $T$. In this case, the argument is a bit more involved.

Let $x \in I_{\tilde{A}_k(n)}$ and write $$Tx = \sumn x_n \left( \sum_{k \in N_n} a_{k,n}e_k\right).$$ Let $p \in \tilde{A}_k(n)$ and let us show $(Tx)_p = 0$. Having in mind that lattice homomorphisms maps different basis elements into expansions with disjoint basis elements, we deduce that
$$
(Tx)_p = x_{\ell} a_{p,\ell},
$$
where $\ell \in \N$. Then, it follows that $p \in N_{\ell}, $ so $\Fi(p) = l.$

Now, since $p \in \tilde{A}_k(n)$, there exists $m \in \N$ such that $p = \Fi^{k+m}(n)$, so $\ell = \Fi(p) = \Fi^{k+m+1}(p)$. Then, $\ell \in \tilde{A}_k(n)$ and since $x \in I_{\tilde{A}_k(n)}$ it follows that $x_{\ell} = 0$. Then $(Tx)_p = 0$ and accordingly, $Tx \in I_{\tilde{A}_k(n)}$. Therefore, $I_{\tilde{A}_k(n)}$ is invariant under $T$ as we wish to prove.
This completes the last part of the proof and the statement of the theorem follows.
\end{proof}

As an immediate corollary we deduce the following result.

\begin{corollary}
Every lattice homomorphism on $\ell^p$,  $1\leq p<  \infty$, has a non-trivial closed invariant subspace which is an ideal.
\end{corollary}

In addition, recall that a positive operator $T$ in a Riesz space is called \emph{interval preserving} whenever $T[0,x]=[0,\, Tx]$ for every positive element $x$. As mentioned in \cite[p. 24]{AA-book}, there are some nice duality properties between interval preserving operators and lattice homomorphisms. Namely, $T$ is a lattice homomorphism if and only if its adjoint $T^*$ is interval preserving.
Accordingly, we have the following

\begin{corollary}
Every interval preserving operator $T$ on $\ell^p$,  $1< p <\infty$, has a non-trivial closed invariant subspace which is an ideal.
\end{corollary}

Lastly, as far as Theorem \ref{lattice invariant} concerns,  we borrow Example 3.5 in \cite{KW} which shows that the result is sharp in the sense that there exist positive operators on $\ell^p$, $1\leq p<\infty$, lacking non-trivial closed invariant sublattices, and hence ideals.

\begin{example}\label{Example1}
Let $X=\ell^p$, $1\leq p<\infty$, and $T$ be the operator defined on $X$ by
$$
(Tx)_n=\left \{
\begin{array}{ll}
x_{n-1}+x_{n+1} & \mbox{ if } n>1,\\
x_2 &  \mbox{ if } n=1.
\end{array}
\right.
$$
There is no non-trivial invariant sublattices of $T$ in $X$.
\end{example}

Indeed, a closer look at Example \ref{Example1} yields that the matrix representation of $T$ is a tridiagonal matrix  where each row (except the first one) has exactly two non-zero entries. This contrasts with the fact that each row of the matrix representation of lattice homomorphisms  has at most one non-zero entry (as proved in Theorem \ref{matrix representation}). This seems to play the key role in the proof of Theorem \ref{lattice invariant}.

\smallskip

Regarding positive operators with tridiagonal matrix, we will see in Section \ref{sec4} that it is possible to provide a characterization of such operators having non-trivial closed invariant ideals.

\medskip

\subsection*{A final remark} In order to conclude Section \ref{sec3}, we observe that the hypotheses on the Banach space $\mathcal{X}$ regarding the order inherited by the unconditional basis is essential and plays a significant role. Indeed, by considering the classical Lebesgue spaces $L^p[0,1]$, $1\leq p<\infty$, one observes that when they are considered as standard Banach lattices, there do exist lattice homomorphisms without non-trivial closed invariant sublattices, and hence ideals  (see \cite[Section 4]{KW}, for instance). Indeed, among the lattice homomorphisms lacking closed invariant sublattices provided by Kitover and Wickstead are certain \emph{Bishop operators}. Recall that if $\alpha$ is an irrational number in the interval $(0,1)$ and $\{\, \cdot \, \}$ stands for the fractional part, the Bishop operator $T_{\alpha}$ is defined by
\begin{align*}
T_\alpha : L^p[0,1] & \longrightarrow L^p[0,1] \\
u(t) & \longmapsto t \cdot u(\{t+\alpha\}).
\end{align*}
Bishop operators were proposed by E. Bishop in the fifties as candidates for operators having no non-trivial closed invariant subspaces, or in other words, operators which might entail  counterexamples for the Invariant Subspace Problem. In 1973, Davie \cite{DAVIE_inv_subs_bishop} showed that there exist Bishop operators having non-trivial hyperinvariant closed subspaces for almost all $\alpha$ (actually, whenever $\alpha$ is a non-Liouville number). It is still an open problem whether all Bishop operators have non-trivial closed invariant subspaces and we refer to \cite{authors2_chamizo_ubis} for results enlarging the class of irrationals $\alpha$'s such that the corresponding Bishop operator $T_{\alpha}$ is known to have non-trivial closed invariant subspaces (see also \cite{Gallardo-Monsalve} for a recent survey on the subject).

\section{Invariant Ideals For Positive Tridiagonal Operators}\label{sec4}

This section is mainly motivated by the operator $T$ in Example \ref{Example1} in the context of the existence of non-trivial closed invariant ideals. As we remarked, the  matrix representation of $T$ is a tridiagonal matrix (note that, in particular,  each row except the first one has exactly two non-zero entries) but $T$ lacks non-trivial closed invariant ideals. As we pointed out, this contrasts with the fact that each row of the matrix representation of lattice homomorphisms has at most one non-zero entry and they do have non-trivial closed invariant ideals.
\smallskip

Our starting point is a theorem due to Grivaux \cite{G02}, which in our context reads as follows:

\begin{thmG}
Let $\mathcal{X}$ be a Banach space with an unconditional basis $\mathcal{E}=\{e_n\}_{n\geq 1}$. If $T$ is a positive operator whose matrix representation respect to  $\mathcal{E}$ is a tridiagonal matrix, then $T$ has a non-trivial closed invariant subspace.
\end{thmG}

Initially, it is natural to ask whether these invariant subspaces can be chosen to be ideals. Nevertheless, as Radjavi and Troitsky showed in \cite[Proposition 5.9]{RT08}, if $Q$ is a linear bounded operator on $\mathcal{X}$ with infinite matrix
$$
Q = \left( \begin{matrix}  0 & * & 0 & 0 & 0 & \cdots \\
		* & 0 & * & 0 & 0 & \cdots \\
		0 & * & 0 & * & 0 & \cdots \\
		0 & 0 & * & 0 & * &  \cdots \\
		\vdots &\vdots &\vdots &\vdots &\vdots & \ddots
		\end{matrix}\right) $$
where $*$ are positive real numbers,  $Q$ has no non-trivial closed invariant ideals. As particular instance, Example \ref{Example1} fits in this scheme.
Indeed, a straightforward consequence of \cite[Proposition 5.9]{RT08} in this context is the following result.

\begin{proposition}\label{tridiagonales} Let $\mathcal{X}$ be a Banach space with an unconditional basis $\mathcal{E}=\{e_n\}_{n\geq 1}$ and $T$ a positive operator whose matrix representation respect to  $\mathcal{E}$ is a tridiagonal matrix $A = (a_{n,m})_{n,m \in \N}$. Assume that both the sub-diagonal $(a_{n+1, n})_{n \in \N}$ and the super-diagonal $(a_{n, n+1})_{n \in \N}$ of $A$ do not have null elements. Then, $T$ has no non-trivial closed invariant ideals.
\end{proposition}

\begin{proof}
Let us denote by $D$ the diagonal operator defined on $\mathcal{E}$ by $D e_n = a_{n,n} e_n$. Let $Q$ the linear operator defined by $Q= T-D$.  Assume, arguing by contradiction, that $T$ has a non-trivial invariant ideal $I$. Observe that $D$ is a central operator, namely, $D$ leaves invariant every ideal of $\mathcal{X}$. Then, $I$ is invariant under $T-D = Q$ which contradicts \cite[Proposition 5.9]{RT08}.
\end{proof}

In addition, as a direct application of Theorem \ref{lattice invariant} in this context, we obtain tridiagonal operators having non-trivial closed invariant ideals.

\begin{proposition}
Let $\mathcal{X}$ be a Banach space with an unconditional basis $\mathcal{E}=\{e_n\}_{n\geq 1}$ and $T$ a positive operator whose matrix representation respect to  $\mathcal{E}$ is a tridiagonal matrix $A = (a_{n,m})_{n,m \in \N}$. Assume  $a_{n,n+1}\, a_{n+1,n} = 0$ for every $n \in \N$. Then, $T$ has a non-trivial closed invariant ideal.
\end{proposition}

\begin{proof}
Consider, as in Proposition \ref{tridiagonales}, the diagonal operator $D$  defined on $\mathcal{E}$ by $D e_n = a_{n,n} e_n$ and $Q = T-D$. Because of the hypothesis $a_{n,n+1}\, a_{n+1,n} = 0$ for every $n \in \N$, either some column is zero or every row of the matrix of $Q$ has at most one positive element. In the former case, for some $j\in\mathbb N$ we have $T e_j=a_{j,j} e_j$, which yields that the subspace generated by $e_j$, which is in fact an ideal, is $T$-invariant. In the latter case, $Q$ is a lattice homomorphism and accordingly, it has a non-trivial closed invariant ideal. Since $D$ is a central operator, $T = Q+D$ has a non-trivial closed invariant ideal, and the proposition follows.
\end{proof}

Last result can be pushed a bit further. In order to state the main result of the section in this sense, we recall a result proved by Radjavi and Troitsky \cite[Proposition 1.2]{RT08} originally stated for operators acting on $\ell^p$ but it can be easily reformulated in the following terms:

\begin{proposition-RT}
Let $\mathcal{X}$ be a Banach space with an unconditional basis $\mathcal{E}=\{e_n\}_{n\geq 1}$ and $T$ a positive operator. Then, $T$ has no non-trivial invariant ideals if and only if for every $i \neq j \in \N$ there exist $n \in \N$ such that $(T^n e_i)_j > 0.$ \label{troitsky}
\end{proposition-RT}

\begin{theorem}\label{theorem positivo}
Let $\mathcal{X}$ be a Banach space with an unconditional basis $\mathcal{E}=\{e_n\}_{n\geq 1}$ and $T$ a positive operator whose matrix representation respect to  $\mathcal{E}$ is a tridiagonal matrix $A = (a_{n,m})_{n,m \in \N}$. Assume there exists a null element either in the sub-diagonal $(a_{n+1, n})_{n \in \N}$ or in  the super-diagonal $(a_{n, n+1})_{n \in \N}$ of $A$. Then, $T$ has a non-trivial closed invariant subspace which is an ideal.
\end{theorem}

\begin{proof}
Assume there exists a null element in the super-diagonal of $A$ , namely $a_{n_0,n_0+1} = 0$ for some $n_0$. The argument for the sub-diagonal of $A$ is similar and it is enough to consider the transpose matrix.

Let us show that the entries $(i,j)$ for $i \in \{1,...,n_0\}$ and $j \geq n_0+1$ are null in the matrix of $T^k$ for every $k \in \N$. As a consequence, we will have $(T^k e_i)_j = 0$ for every $k \in \N$ and the indicated indices, so by the previous Proposition the result will follow. We argue by induction on $k$.

Since $a_{n_0,n_0+1} = 0$ the induction hypothesis follows for $k=1$. Now, assume the hypothesis for $k \in \N$. Consider $A^{k+1}$ and let $i \in \{1,...,n_0\}$ and $j \geq n_0+1.$  Clearly  $A^{k+1} = A\,A^k$ and the $(i,j)$ element of the matrix $A^{k+1}$ is the product of the $i-$th row of $A$ by the $j-$th column of $A^k.$ Observe all the positive elements of the $i$-th row of $A$ are on the first $n_0$ coordinates of the vector. Moreover, by induction, the first $n_0$ coordinates of the $j$-th column of $A^k$ are null. Hence, the product of the $i$-th row and the $j$-th column is zero, and therefore the element $(i,j)$ of $A^{k+1}$ is zero, which finishes proof.
\end{proof}

As a byproduct of the previous results, we have the characterization for tridiagonal operators lacking non-trivial closed invariant ideals:

\begin{corollary}
Let $\mathcal{X}$ be a Banach space with an unconditional basis $\mathcal{E}=\{e_n\}_{n\geq 1}$ and $T$ a positive operator whose matrix representation respect to  $\mathcal{E}$ is a tridiagonal matrix $A$.  Then, $T$ has no non-trivial closed invariant ideals if and only if both the sub-diagonal and the super-diagonal of $A$ have no null elements.
\end{corollary}

\end{document}